\documentclass[10pt,a4paper]{article}

\usepackage{amssymb,amsmath,mathrsfs,amsthm}
\usepackage{verbatim}
\usepackage[a4paper,includeheadfoot,left=30mm,right=30mm,top=25mm,bottom=30mm]{geometry} 
\usepackage{xcolor,graphicx}

\newtheorem{thm}{Theorem}[section]
\newtheorem{lem}[thm]{Lemma}
\newtheorem{prop}[thm]{Proposition}

\newtheorem{defn}[thm]{Definition}

\theoremstyle{definition}

\newtheorem{rmk}{Remark}
\theoremstyle{remark}

\newcommand{\hsforall}{\hspace{1mm}\forall\hspace{1mm}}						 
\renewcommand{\Re}{\operatorname*{Re}}                             
\renewcommand{\Im}{\operatorname*{Im}}                             
\renewcommand{\d}{\ensuremath{\,\mathrm{d}}}							         
\newcommand{\Mspacer}{\hspace{0.5mm}}                              
\newcommand{\M}[3]{#1_{#2\Mspacer#3}}                              
\newcommand{\Msup}[4]{#1_{#2\Mspacer#3}^{#4}}                      
\newcommand{\Msups}[5]{#1_{#2\Mspacer#3}^{#4\Mspacer#5}}           
\renewcommand{\geq}{\geqslant}                                     
\renewcommand{\leq}{\leqslant}                                     
\renewcommand{\epsilon}{\varepsilon}                               
\newcommand{\BE}{\begin{equation}}                                 
\newcommand{\EE}{\end{equation}}                                   
\newcommand{\BES}{\begin{equation*}}                               
\newcommand{\EES}{\end{equation*}}                                 
\newcommand{\BP}{\begin{pmatrix}}                                  
\newcommand{\EP}{\end{pmatrix}}                                    
\newcommand{\R}{\mathbb{R}}                                        
\newcommand{\C}{\mathbb{C}}                                        

\newcommand{\superscript}[1]{\ensuremath{^{\textrm{#1}}}}

\newcommand{\Thns}[0]{\superscript{th}}
\newcommand{\Th}[0]{\Thns~}

\def\clap#1{\hbox to 0pt{\hss#1\hss}}

\numberwithin{equation}{section}

\title{Evolution PDEs and augmented eigenfunctions. \\ Half-line}
\author{
B. Pelloni$^1$ and D. A. Smith$^2$
\\
\footnotesize $^1$ Department of Mathematics \& Statistics, University of Reading, \\ \footnotesize Whiteknights, PO Box 220, Reading RG6 6AX UK \\
\footnotesize $^2$
\footnotesize Department of Mathematical Sciences, University of Cincinnati, OH \\ \footnotesize Current address: Department of Mathematics, University of Michigan, \\ \footnotesize Ann Arbor, MI 48109-1043 USA \\
\footnotesize email\textup{: \texttt{daasmith@umich.edu}}
}

\begin{document}
\maketitle

\begin{abstract}
The  solution of an initial-boundary value problem for a linear evolution partial differential equation posed on the half-line can be represented in terms of an integral in the complex (spectral) plane. This representation is obtained by the {\em unified transform} introduced by Fokas in the 90's. On the other hand, it is known that many initial-boundary value problems can be solved via a classical transform pair,  constructed via the spectral analysis of the associated spatial operator. For example, the Dirichlet problem for the heat equation can be solved by applying the Fourier sine transform pair. However, for many other initial-boundary value problems there is {\em no} suitable transform pair in the classical literature. Here we pose and answer two related questions: Given any well-posed initial-boundary value problem, does there exist a (non-classical) transform pair suitable for solving that problem? If so, can this transform pair be constructed via the spectral analysis of a differential operator? The answer to both of these questions is positive and given in terms of {\em augmented eigenfunctions}, a novel class of spectral functionals. These are eigenfunctions of a suitable differential operator in a certain generalised sense, they provide an effective spectral representation of the operator, and are associated with a transform pair suitable to solve the given initial-boundary value problem.
\end{abstract}

\subsubsection*{AMS MSC2010}35P10 (primary), 35C15, 35G16, 47A70 (secondary).

\section{Introduction} \label{sec:Intro}

In this paper we consider  initial-boundary value problems (IBVP) for  linear evolution constant-coefficient partial differential equations (PDE). The classical transform pairs used to solve problems of this kind  are based on the representation of the given initial condition  as an expansion in a complete system of (generalised) eigenfunctions of an appropriate differential operator, namely the operator associated with the spatial part of the IBVP. (We assume the usual Hilbert space structure, inherited from $L^2$, on the underlying function space.)  

\noindent
This method relies crucially upon two properties, namely

(1) the completeness of the spectral system;

(2)  the convergence of the expansion of the initial condition in the system. 

\noindent
It is not surprising that such approaches fail, even for simple high-order problems, as soon as the differential operator is non-self-adjoint~\cite{Pel2005a}. 

On the other hand, problems of this type can be solved using the unified transform method, introduced by Fokas in the late 90's~\cite{Fok1997a,FP2001a,FS1999a}. Indeed, a representation of the solution, assuming it exists and is unique, can be given by Fokas' approach regardless of order or complexity of boundary conditions. Moreover, the unified method was used by the present authors to obtain well-posedness criteria~\cite{Pel2004a,Smi2012a}.

\smallskip
In this paper we interpret the solution representation given by the unified transform method of Fokas in terms of integral transform pairs, and discuss the spectral meaning of these transform pairs. Herein, we provide results for initial-boundary value problems on a semi-infinite domain and the associated differential operators. This discussion complements the picture presented in~\cite{FS2013a}, where problems posed on a finite spatial interval are studied. In~\cite{Smi2014a}, the results of the present work are compared and contrasted with the finite interval results.

More specifically, we derive transform pairs for boundary value problems for constant coefficient PDEs in two independent variables, of the general form 
\BE
\frac{\partial}{\partial t}q(x,t)+a(-i)^n\frac{\partial^n}{\partial x^n}q(x,t)=0,\quad a\in\C,
\label{PDEin}
\EE
posed on the half line $(0,\infty)$. It is known \cite{FS1999a} what boundary conditions must be imposed to obtain a problem that admits a unique solution, and we consider only such well-posed problems. 
Any initial-boundary value problem for (\ref{PDEin})  is naturally associated to the study of a differential operator  such as the operator $S$ defined below by~\eqref{eqn:defnS}, complemented with the appropriate boundary conditions.  The spectral  representation of this operator and its diagonalisation are described by introducing a more general type of eigenfunctions, that we call {\em augmented eigenfunctions}, and that can be read off  the integral representation of the PDE problem. In addition, the {\em completeness} of the eigenfunction family and the {\em convergence} of the associated expansion can be obtained through the PDE results obtained using the unified transform method of Fokas.

\subsubsection*{The main illustrative examples}
Throughout the paper, we will use two examples to illustrate the main results. To wit, consider the following initial-boundary value problems:
\newline\noindent{\bf Problem 1}: the linearised Korteweg-de Vries (LKdV) equation
\begin{subequations} \label{eqn:introIBVP.1}
\begin{align} \label{eqn:introIBVP.1:PDE}
q_t(x,t) + q_{xxx}(x,t) &= 0 & (x,t) &\in (0,\infty)\times(0,T), \\
q(x,0) &= f(x) & x &\in [0,\infty), \\
q(0,t) &= 0 & t &\in [0,T].
\end{align}
\end{subequations}
\noindent{\bf Problem 2}: the reverse-time linearised Korteweg-de Vries equation
\begin{subequations} \label{eqn:introIBVP.2}
\begin{align} \label{eqn:introIBVP.2:PDE}
q_t(x,t) - q_{xxx}(x,t) &= 0 & (x,t) &\in (0,\infty)\times(0,T), \\
q(x,0) &= f(x) & x &\in [0,\infty), \\
q(0,t) = q_x(0,t) &= 0 & t &\in [0,T].
\end{align}
\end{subequations}
It is shown in~\cite{Fok1997a, FS1999a} that these problems are well-posed. The solution of problem~1 can be expressed in the form
\begin{subequations} \label{eqn:introIBVP.solution.2}
\BE
q(x,t) = \frac{1}{2\pi}\int_{\Gamma_1} e^{i\lambda x + i\lambda^3t} \zeta_1(\lambda;f)\d\lambda + \frac{1}{2\pi}\int_{\Gamma_0} e^{i\lambda x + i\lambda^3t} \hat{f}(\lambda)\d\lambda,
\EE
where $\Gamma_1$ is the oriented boundary of the domain $\{\lambda\in\C^+:\Re(-i\lambda^3)<0\}$ perturbed away from $0$, as shown in figure~\ref{fig:3o-cont}. Similarly, the solution of problem~2 can be expressed as
\begin{multline}
q(x,t) = \frac{1}{2\pi}\int_{\Gamma_1} e^{i\lambda x + i\lambda^3t} \zeta_1(\lambda;f)\d\lambda + \frac{1}{2\pi}\int_{\Gamma_2} e^{i\lambda x + i\lambda^3t} \zeta_2(\lambda;f)\d\lambda \\
+ \frac{1}{2\pi}\int_{\Gamma_0} e^{i\lambda x + i\lambda^3t} \hat{f}(\lambda)\d\lambda,
\end{multline}
\end{subequations}
where $\Gamma_1$ and $\Gamma_2$ are the connected components of the oriented boundary of the domain $\{\lambda\in\C^+:\Re(-i\lambda^3)<0\}$, perturbed away from $0$, as shown in figure~\ref{fig:3o-cont}. In both problems, the contour $\Gamma_0$ is $\R$ perturbed away from $0$ along a small semicircular arc in $\C^+$. The function $\hat{f}$ denotes the Fourier transform of $f(x)$, extended to a function on $\C^+$, given by
\BE
\hat{f}(\lambda) = \int_0^\infty e^{-i\lambda x}f(x)\d x, \qquad \lambda\in\C^+,
\EE
and the functions $\zeta_j(\lambda)$, $j=1,2$, appearing in the solution representations are defined as follows:
\begin{align}\label{eqn:introIBVP.Zeta.1}
{\bf Problem ~1}: \quad \zeta_1(\lambda) &= - \alpha\hat{f}(\alpha\lambda) - \alpha^2\hat{f}(\alpha^2\lambda), & \lambda &\in \Gamma_1. \\
\label{eqn:introIBVP.DeltaZeta.2}
{\bf Problem ~2}: \quad \zeta_1(\lambda) &= \hat{f}(\alpha^2\lambda), & \lambda &\in \Gamma_1, \\
\zeta_2(\lambda) &= \hat{f}(\alpha\lambda),   & \lambda &\in \Gamma_2.
\end{align}
where $\alpha$ is a cube root of unity:
\BE
\alpha =e^{2\pi i/3}.
\label{alpha}\EE
In the sequel, we will consider the analytic extension of the functions $\zeta_j$ to appropriate closed sectors, without further comment. 
\begin{figure}
\begin{center}
\includegraphics{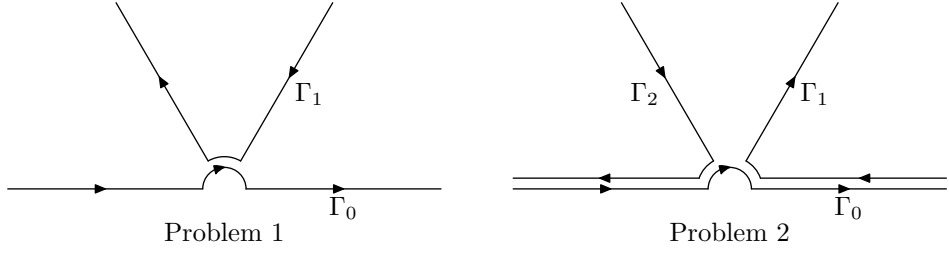}
\caption{Contours for the LKdV and reverse-time LKdV equations.}
\label{fig:3o-cont}
\end{center}
\end{figure}

\subsubsection*{Transform pair}

It is well-known that the half-line homogeneous Dirichlet problem for the heat equation $q_t=q_{xx}$ on $(0,\infty)$ may be solved by using the Fourier sine transform pair. The solution is
\BE
q(x,t) = f\left[e^{-\lambda^2t}F[f](\lambda)\right](x),
\EE
where the direct and invese transforms are defined as follows:
\begin{align} \label{eqn:Sine.Transform.Pair}
f(x) &\mapsto F[f](\lambda): & F[f](\lambda) &= \frac{2}{\pi} \int_0^\infty \sin(\lambda x) f(x)\d x, & \lambda &\in [0,\infty), \\
F(\lambda) &\mapsto f[F](x): & f[F](x)       &= \int_0^\infty \sin(\lambda x) F(\lambda)   \d\lambda, & x &\in [0,\infty).
\end{align}
Similarly, the half-line homogeneous Neumann problem for the heat equation is solved with the Fourier cosine transform pair. However, for higher order IBVP, the standard sine, cosine and exponential Fourier transforms are inadequate. Moreover, classical separation of variables techniques often do not yield the requisite transform pairs, in contrast with what one may expect based on the second order examples. This is due to the fact that the boundary conditions may be non-self-adjoint, or non-separable~\cite{FS2012a}.

It turns out that the unified transform method of Fokas provides an algorithm for constructing a transform pair tailored to a given initial-boundary value problem even in such cases. For example, the integral representations~\eqref{eqn:introIBVP.solution.2} give rise to the following transform pair, which is tailored for solving problems~1 and~2:
\begin{subequations} \label{eqn:introTrans.1.1}
\begin{align} \label{eqn:introTrans.1.1a}
f(x) &\mapsto F[f](\lambda): & F[f](\lambda) &= \int_0^\infty \phi^j(x,\lambda)f(x)\d x & \lambda &\in \Gamma_j, \\ \label{eqn:introTrans.1.1b}
F(\lambda) &\mapsto f[F](x): & f[F](x) &= \sum_{j=0}^{1\text{ or }2} \int_{\Gamma_j} e^{i\lambda x} F(\lambda) \d\lambda, & x &\in [0,\infty),
\end{align}
where 
\begin{align} \label{eqn:introTrans.1.1bb}
\phi^0(x,\lambda)     &= \frac{1}{2\pi}e^{-i\lambda x}, & \lambda &\in \Gamma_0, \\
\label{eqn:introTrans.1.1c}
\mbox{(for problem~1)}\quad \phi^1(x,\lambda)     &= \frac{-1}{2\pi} \left[ \alpha e^{-i\alpha\lambda x} + \alpha^2 e^{-i\alpha^2\lambda x} \right], & \lambda &\in \Gamma_1, \\ 
\label{eqn:introTrans.1.1d}
\mbox{(for problem~2)}\quad \phi^{1}(x,\lambda)   &= \frac{1}{2\pi} e^{-i\alpha^2\lambda x}, & \lambda &\in \Gamma_1, \\ \label{eqn:introTrans.1.1e}
\phi^{2}(x,\lambda)   &= \frac{1}{2\pi} e^{-i\alpha\lambda x},   & \lambda &\in \Gamma_2,
\end{align}
\end{subequations}
with $\alpha$ given by (\ref{alpha}).

The validity of these transform pairs is established in section~\ref{sec:Transforms.valid}. In  section~\ref{sec:Transform.Method} we prove that the solution of problems~1 and~2 is given by
\BE \label{eqn:introIBVP.solution.transform.1}
q(x,t) = f\left[e^{i\lambda^3t}F[f](\lambda)\right](x).
\EE

The transform pairs~\eqref{eqn:introTrans.1.1} are much less symmetric than the Fourier sine transform pair~\eqref{eqn:Sine.Transform.Pair}. This is not entirely surprising. One would expect the direct transform to be related, in some way, to the spectral representation of the spatial differential operator, while the inverse transform should be associated with the spectral representation of the adjoint operator. As the underlying spatial differential operator is not self-adjoint, these representations will generally be different. However what {\em is} surprising is the fact, described in the sequel, that these transforms are not constructed in the usual way in terms of some associated spectral objects, as the expansions resulting from such a construction may fail to be convergent.  

\subsubsection*{The generalised spectral representation of Gelfand}

To illustrate the need for introducing a generalised notion of spectral representation we start with a classical example. Let $\mathcal{S}[0,\infty)$ denote the Schwartz space of half-line restrictions of rapidly-decaying functions. Suppose we seek eigenfunctions, in the usual sense, of the spatial differential operator $S$ associated with the Dirichlet problem for the heat equation, given by
\BE
(Sf)(x) = -f''(x), \qquad \hsforall f\in \mathcal{S}[0,\infty) \mbox{ such that } f(0) = 0.
\label{Ssine}\EE
Thus we seek a function $f=f_\lambda$ such that  $-f_\lambda''(x) = \lambda^2f_\lambda(x)$, $\lambda \in\R$. This  implies $f_\lambda(x) = Ae^{i\lambda x} + Be^{-i\lambda x}$ and the boundary condition yields $B=-A$, so that 
\BE
f_\lambda(x) = A'\sin(\lambda x).
\EE
But, for $f\in \mathcal{S}[0,\infty)$, we must have $A'=0$ so there are no nonzero eigenfunctions of $S$.

In~\cite{GS1967a,GV1964a}, Gel'fand and coauthors described the concept of eigenfunctionals or \emph{generalised eigenfunctions}. Namely, they sought functionals $F[\cdot](\lambda)\in(\mathcal{S}[0,\infty))'$ such that for every $\lambda\in\R$, 
\BE \label{eqn:DHeat.Eigenfunctionals1}
F[Sf](\lambda) = \lambda^2F[f](\lambda).
\EE
The above relation holds provided
\BE \label{eqn:DHeat.Eigenfunctionals2}
F[f](\lambda) = \frac{2}{\pi}\int_0^\infty\sin(\lambda x)f(x)\d x.
\EE
The functionals $F[\cdot](\lambda)$ are the generalised eigenfunctions of the operator $S$ defined by (\ref{Ssine}). Note that the generalised eigenfunction corresponding to a given $\lambda\in\R$ is precisely the evaluation at $\lambda$ of the direct transform used to solve the corresponding IBVP; generalised eigenfunctions are therefore very natural spectral objects.

The primary achievement of~\cite[chapter~1]{GV1964a} is to elucidate how generalised eigenfunctions are a relevant spectral object: they provide a spectral representation of any self-adjoint linear operator,
corresponding to a spectral parameter $\lambda\in\R$ that can be interpreted as a continuous eigenvalue. Indeed, Gel'fand still uses the term ``eigenvalue'' for the continuous spectral parameter and we will follow his convention. 

Note that Gel'fand's formulation of generalised eigenfunctions requires self-adjointness of the linear differential operator for completeness results. This certainly holds for the Dirichlet heat operator but the non-self-adjoint boundary conditions considered in problems~1 and~2 preclude an application of the spectral theory presented in~\cite[chapter~1]{GV1964a}. 

\subsubsection*{Augmented eigenfunctions}
In light of the discussion in the previous section, it is natural to ask whether the transform pairs~\eqref{eqn:introTrans.1.1}, which were derived through the unified transform method of Fokas to solve problems~1 and~2, have similar spectral meanings to the sine transform. In what follows, we describe the abstract notion of~\emph{augmented eigenfunctions}.

\begin{defn} \label{defn:Aug.Eig}
Let $I\subset \R$ be open and let $C$ be a topological vector space of functions defined on the closure of $I$, with sufficient smoothness and decay conditions. 

\noindent
Let $\Phi\subseteq C$ and let $L:\Phi\to C$ be a linear differential operator of order $n$. 

\noindent
Let $\Gamma$ be an oriented contour in $\C$ and let $E=\{E[\cdot](\lambda):\lambda\in\Gamma\}$ be a family of functionals $E[\cdot](\lambda)\in C'$. Suppose there exist corresponding \emph{remainder} functionals $R[\cdot](\lambda)\in C'$ 
such that
\BE \label{eqn:defnAugEig.AugEig}
E[L\phi](\lambda) = \lambda^n E[\phi](\lambda) + R[\phi](\lambda), \qquad\hsforall\phi\in\Phi, \hsforall \lambda\in\Gamma.
\EE

If
\BE \label{eqn:defnAugEig.Control1}
\int_\Gamma e^{i\lambda x} R[\phi](\lambda)\d\lambda = 0, \qquad \hsforall \phi\in\Phi, \hsforall x\in I,
\EE
then we say $E$ is a family of \emph{type~\textup{I} augmented eigenfunctions} of $L$ up to integration along $\Gamma$.

If
\BE \label{eqn:defnAugEig.Control2}
\int_\Gamma \frac{e^{i\lambda x}}{\lambda^n}  R[\phi](\lambda)\d\lambda = 0, \qquad \hsforall \phi \in\Phi, \hsforall x\in I, 
\EE
then we say $E$ is a family of \emph{type \textup{II}~augmented eigenfunctions} of $L$ up to integration along $\Gamma$.
\end{defn}

Note that we cannot restrict the spectral parameter to real values, as the resulting expansion may then fail to converge, as in the sine example above. In the definition above the crucial spectral parameter takes the form $\lambda^n$. Hence in general, even when $\lambda^n\in\R$, the usual spectral parameters given by the $n$\Th roots $\lambda$ are complex, and the eigenfunctionals involve complex integration.
This mirrors the situation with representing the solution of the initial-boundary value problem as an integral along a complex contour, and is a manifestation of the lack of symmetry in the operator.

\begin{rmk} \label{rmk:Pseudospectra}
The remainder functional $R[\cdot](\lambda)$ appears also 
in the theory of pseudospectra~\cite{ET2005a}. In that context, it is required that the norm of  $R[\cdot](\lambda)$ be less than some small value. Our definition serves a different application, and rather than a small norm, we require that the integral of $\exp(i\lambda x)R[\phi](\lambda)$ along the contour $\Gamma$ vanishes. 
\end{rmk}

It will be shown in section~\ref{sec:Spectral} that $\{F[\cdot](\lambda):\lambda\in\Gamma\}$ is a family of type~II augmented eigenfunctions of the differential operator representing the spatial part of problem~1 or problem~2, with eigenvalue $\lambda^3$. It will also be shown that $\{F[\cdot](\lambda):\lambda\in\Gamma_1\}$ is a family of type~I augmented eigenfunctions of the spatial differential operator in problem~1; the corresponding functionals arising in problem~2, $\{F[\cdot](\lambda):\lambda\in\Gamma_1\cup\Gamma_2\}$, do not form a family of type~I augmented eigenfunctions.

\subsubsection*{Spectral representation of non-self-adjoint operators}

The definition of augmented eigenfunctions, in contrast to the generalized eigenfunctions of Gel'fand and Vilenkin~\cite[section 1.4.5]{GV1964a}, allows the occurrence of remainder functionals. However, the contribution of these remainder functionals is eliminated by multiplying by the Fourier kernel and integrating over $\Gamma$. Indeed, integrating equation~\eqref{eqn:defnAugEig.AugEig} over $\Gamma$ with respect to the Fourier kernel gives rise to a non-self-adjoint analogue of the spectral representation of an operator.

\begin{defn} \label{defn:Complete}
We say that $E=\{E[\cdot](\lambda):\lambda\in\Gamma\}$ is a \emph{complete} family of functionals $E[\cdot](\lambda)\in C'$ if
\BE \label{eqn:defnGEInt.completecriterion}
\phi\in\Phi \mbox{ and } E[\phi](\lambda) = 0 \hsforall \lambda\in\Gamma \quad \Rightarrow \quad \phi=0.
\EE
\end{defn}

We now define a spectral representation of the non-self-adjoint differential operators we study in this paper.

\begin{defn} \label{defn:Spect.Rep.II}
Suppose that $E=\{E[\cdot](\lambda):\lambda\in\Gamma\}$ is a system of type~\textup{II} augmented eigenfunctions of $L$ up to integration over $\Gamma$, and that
\BE \label{eqn:Spect.Rep.defnI.conv}
\int_\Gamma e^{i\lambda x} E[\phi](\lambda) \d\lambda \mbox{\textnormal{ converges }} \hsforall \phi\in\Phi, \hsforall x\in I.
\EE
Furthermore, assume that $E$ is a complete system in the sense of definition~\ref{defn:Complete}. Then we say that $E$ provides a \emph{spectral representation} of $L$ in the sense that
\BE \label{eqn:Spect.Rep.I}
\int_\Gamma e^{i\lambda x} \frac{1}{\lambda^n} E[L\phi](\lambda) \d\lambda = \int_\Gamma e^{i\lambda x} E[\phi](\lambda) \d\lambda \qquad \hsforall \phi\in\Phi, \hsforall x\in I.
\EE
\end{defn}

\begin{defn} \label{defn:Spect.Rep.I.II}
Suppose that $E_{(\mathrm{I})}=\{E[\cdot](\lambda):\lambda\in\Gamma_{(\mathrm{I})}\}$ is a system of type~\textup{I} augmented eigenfunctions of $L$ up to integration over $\Gamma_{(\mathrm{I})}$ and that
\BE \label{eqn:Spect.Rep.defnI.II.conv1}
\int_{\Gamma_{(\mathrm{I})}} e^{i\lambda x} E[L\phi](\lambda) \d\lambda \mbox{\textnormal{ converges }} \hsforall \phi\in\Phi, \hsforall x\in I.
\EE
Suppose also that $E_{(\mathrm{II})}=\{E[\cdot](\lambda):\lambda\in\Gamma_{(\mathrm{II})}\}$ is a system of type~\textup{II} augmented eigenfunctions of $L$ up to integration over $\Gamma_{(\mathrm{II})}$ and that
\BE \label{eqn:Spect.Rep.defnI.II.conv2}
\int_{\Gamma_{(\mathrm{II})}} e^{i\lambda x} E[\phi](\lambda) \d\lambda \mbox{\textnormal{ converges }} \hsforall \phi\in\Phi, \hsforall x\in I.
\EE
Furthermore, assume that $E=E_{(\mathrm{I})}\cup E_{(\mathrm{II})}$ is a complete system in the sense of definition~\ref{defn:Complete}. Then we say that $E$ provides a \emph{spectral representation} of $L$ in the sense that
\begin{subequations} \label{eqn:Spect.Rep.II}
\begin{align} \label{eqn:Spect.Rep.II.1}
\int_{\Gamma_{(\mathrm{I})}} e^{i\lambda x} E[L\phi](\lambda) \d\lambda &= \int_{\Gamma_{(\mathrm{I})}} \lambda^n e^{i\lambda x} E[\phi](\lambda) \d\lambda & \hsforall \phi &\in\Phi, \hsforall x\in I, \\ \label{eqn:Spect.Rep.II.2}
\int_{\Gamma_{(\mathrm{II})}} \frac{1}{\lambda^n} e^{i\lambda x} E[L\phi](\lambda) \d\lambda &= \int_{\Gamma_{(\mathrm{II})}} e^{i\lambda x} E[\phi](\lambda) \d\lambda & \hsforall \phi &\in\Phi, \hsforall x\in I.
\end{align}
\end{subequations}
\end{defn}

Completeness is an essential component of any definition of a spectral representation; see Gel'fand's definition~\cite{GV1964a}. Indeed, otherwise, for some nonzero $\phi\in\Phi$, equation~\eqref{eqn:Spect.Rep.I} is trivially $0=0$. Crucially, \emph{it is possible to obtain the requisite completeness and convergence results by studying the IBVP} associated with the operator.

\begin{rmk} \label{rmk:defnAugEig.ev.power}
Our definitions above are given for $L$ a linear differential operator equal to its principal part. Note however that the unified transform method of Fokas can be applied to problems where the associated spatial differential operator has an arbitrary polynomial~\cite{FP2001a}, and even rational~\cite{DV2011a,FP2005b}, characteristic. It is reasonable to expect that the definition of augmented eigenfunctions and the theory presented in this paper could be extended at least to these cases, albeit with significant notational complication. In order to simplify the presentation, we avoid these complications in the present work.

The unified transform method has not yet been implemented for partial differential equations with variable coefficients. To extend the results of this paper to operators with variable coefficients would require either such an extension of the unified transform method, or the development of a new approach to prove the theorems presented below.
\end{rmk}

\subsubsection*{Results and organisation of paper}

The two problems above are each typical of a class of IBVP. Indeed, for each well-posed half-line IBVP, we can always use the unified transform method of Fokas to construct a transform-inverse transform pair tailored to the problem, where the forward transform can be viewed as a family of type~II augmented eigenfunctions. Moreover, these type~II augmented eigenfunctions provide a spectral representation of the associated differential operator in the sense of definition~\ref{defn:Spect.Rep.II}. These results are the contents of proposition~\ref{prop:Transform.Method:General:2} and theorem~\ref{thm:Diag.2}.

If, as in problem~1 but not problem~2,
\BE \label{eqn:Contours.for.type.I}
\mbox{the contour } \bigcup_{j\geq1}\Gamma_j \mbox{ has no semi-infinite component lying on } \R,
\EE
then the family of functionals
\BE
\left\{F[\cdot](\lambda):\lambda\in\bigcup_{j\geq1}\Gamma_j\right\}
\EE
is a family of type~I augmented eigenfunctions. The class of problems for which statement~\eqref{eqn:Contours.for.type.I} holds is described in theorem~\ref{thm:Diag.5}; it is also shown that the augmented eigenfunctions provide a spectral representation in the sense of definition~\ref{defn:Spect.Rep.I.II}. Since $\{F[\cdot](\lambda):\lambda\in\Gamma_0\}$ is never a family of type~I augmented eigenfunctions, $S$ cannot have a spectral representation provided solely by type~I augmented eigenfunctions.

In section~\ref{sec:Transforms.valid}, we establish that the integral transforms~\eqref{eqn:introTrans.1.1} are indeed valid transform-inverse transform pairs and then extend this result to the general case. Namely, we define a general $n$\Th order operator $S$, with arbitrary linear boundary conditions. We also define associated well-posed IBVP and the transform pairs used to solve these IBVP. To complete section~\ref{sec:Transforms.valid}, we prove that the general integral transforms also give valid transform-inverse transform pairs.
In section~\ref{sec:Transform.Method}, we show that the transform pair may be used to solve the IBVP, first for the example problems~1 and~2, and then in general.
Finally, in  section~\ref{sec:Spectral}, we show that the forward transforms may be viewed as augmented eigenfunctions of the operator $S$ and prove results on the spectral representation of $S$ via its augmented eigenfunctions.

\section{Validity of transform pairs} \label{sec:Transforms.valid}

In section~\ref{ssec:Transforms.valid:LKdV} we will show the validity of the transform pairs defined by equations~\eqref{eqn:introTrans.1.1}. In section~\ref{ssec:Transforms.valid:LKdV:General} we derive an analogous transform pair for a general IBVP. In section~\ref{ssec:GeneralValidity}, we establish the validity of the general transform pair.

Throughout this paper, we work in the space of half-line restrictions of smooth, compactly supported functions,
\BE \label{eqn:defn.C}
C = C^\infty_0[0,\infty) = \{ f|_{[0,\infty)} : f\in C^\infty_0(\R) \}.
\EE
The unified transform method has been shown to be valid on the Schwartz space $\mathcal{S}[0,\infty)$, and even on spaces with lower regularity~\cite{FS1999a}. In order to ensure the usual Fourier transform is defined everywhere on $\Gamma_0$, some additional decay beyond Schwartz is required (see equations~\eqref{eqn:introIBVP.solution.2} and figure~\ref{fig:3o-cont}). We may recover the usual space of validity of the unified method by observing that $C$ is dense when considered as a subspace of $\mathcal{S}[0,\infty)$. See also remark~\ref{rmk:Gamma0.shape}.

\subsection{Linearized KdV} \label{ssec:Transforms.valid:LKdV}

\begin{prop} \label{prop:Transforms.valid:LKdV:2}
Let $F[f](\lambda)$ and $f[F](x)$ be given by equations~\emph{\eqref{eqn:introTrans.1.1a}--\eqref{eqn:introTrans.1.1c}}.
For all $f\in C$ such that $f(0)=0$ and for all $x\in(0,\infty)$, we have
\BE \label{eqn:Transforms.valid:LKdV:prop2:prob1}
f[F[f]](x) = f(x).
\EE
Let $F[f](\lambda)$ and $f[F](x)$ be given by equations~\emph{\eqref{eqn:introTrans.1.1a}--\eqref{eqn:introTrans.1.1bb}},~\eqref{eqn:introTrans.1.1d} and~\eqref{eqn:introTrans.1.1e}.
For all $f\in C$ such that $f(0)=f'(0)=0$ and for all $x\in(0,\infty)$,
\BE \label{eqn:Transforms.valid:LKdV:prop2:prob2}
f[F[f]](x) = f(x).
\EE
\end{prop}

\begin{proof}
The definition of the transform pair~\eqref{eqn:introTrans.1.1a}--\eqref{eqn:introTrans.1.1c} implies
\BE \label{eqn:Transforms.valid:LKdV:prop2:proof.1}
f[F[f]](x) = \frac{1}{2\pi}\int_{\Gamma_1} e^{i\lambda x} \zeta_1(\lambda) \d\lambda + \frac{1}{2\pi}\int_{\Gamma_0} e^{i\lambda x} \hat{f}(\lambda) \d\lambda,
\EE
where $\zeta$ is given by equation~\eqref{eqn:introIBVP.Zeta.1} and the contours $\Gamma_1$ and $\Gamma_0$ are shown in figure~\ref{fig:3o-cont}.

As $\lambda\to\infty$ from within the closed sector $\{\lambda:\frac{\pi}{3}\leq\arg(\lambda)\frac{2\pi}{3}\}$, the exponentials $e^{-i\alpha\lambda}$ and $e^{-i\alpha^2\lambda}$ are bounded. Integration by parts and the boundary condition yields $\hat{f}(\alpha\lambda)$, $\hat{f}(\alpha^2\lambda) = O(\lambda^{-2})$ and these Fourier transforms are holomorphic in the same sector. Hence, by Jordan's lemma, the integral over $\Gamma_1$ vanishes. The integrand $e^{i\lambda x}\hat{f}(\lambda)$ is holomorphic hence admits a deformation of the contour $\Gamma_0$ onto $\R$. The validity of the usual Fourier transform on $\mathcal{S}[0,\infty)$ completes the proof.

The proof for the transform pair~\eqref{eqn:introTrans.1.1a}--\eqref{eqn:introTrans.1.1bb},~\eqref{eqn:introTrans.1.1d} and~\eqref{eqn:introTrans.1.1e} is similar.
\end{proof}

\subsection{General case: definition of transform pair} \label{ssec:Transforms.valid:LKdV:General}

\subsubsection*{Spatial differential operator}

Let $n\geq2$ and $N\in\{n/2,(n-1)/2,(n+1)/2\}$ be integers. Let $B_j:C\to\mathbb{C}$ be the following linearly independent boundary forms
\BE
B_j\phi = \sum_{k=0}^{n-1} \M{b}{j}{k}\phi^{(k)}(0), \quad j\in\{1,2,\ldots,N\},
\EE
with boundary coefficients $\M{b}{j}{k} \in \R$. The integer $N$ is defined by equation~\eqref{eqn:defn.N}. Let
\BE
\Phi=\{\phi\in C:B_j\phi=0\hsforall j\in\{1,2,\ldots,N\}\}
\EE
and let $\{B_j^\star:j\in\{1,2,\ldots,n-N\}\}$ be a set of adjoint boundary forms with adjoint boundary coefficients $\Msup{b}{j}{k}{\star} \in \R$. Let $S:\Phi\to C$ be the differential operator defined by
\BE \label{eqn:defnS}
S\phi(x)=(-i)^n\frac{\d^n\phi}{\d x^n}(x).
\EE
Then $S$ is formally self-adjoint but, in general, does not admit a self-adjoint extension because, in general, $B_j\neq B_j^\star$. Indeed, adopting the notation
\BE
[\phi\psi](x) = (-i)^n\sum_{j=0}^{n-1}(-1)^j\left(\phi^{(n-1-j)}(x)\overline{\psi}^{(j)}(x)\right),
\EE
of~\cite[section~11.1]{CL1955a} and integrating by parts, we find
\BE \label{eqn:S.not.s-a}
((-i\d/\d x)^n\phi,\psi) = - [\phi\psi](0) + (\phi,(-i\d/\d x)^n\psi), \qquad \hsforall \phi,\psi\in C.
\EE
If $\phi\in\Phi$, then $\psi$ must satisfy the adjoint boundary conditions in order for $[\phi\psi](0) = 0$.

\subsubsection*{Initial-boundary value problem}

Associated with $S$ and the constant $a\in\C$, we define the following homogeneous initial-boundary value problem:
\begin{subequations} \label{eqn:IBVP}
\begin{align} \label{eqn:IBVP.PDE}
(\partial_t + aS)q(x,t) &= 0 & \hsforall (x,t) &\in (0,\infty)\times(0,T), \\ \label{eqn:IBVP.IC}
q(x,0) &= f(x) & \hsforall x &\in [0,\infty), \\ \label{eqn:IBVP.BC}
q(\cdot,t) &\in \Phi & \hsforall t &\in [0,T],
\end{align}
\end{subequations}
where $f\in\Phi$ is arbitrary.
Such a problem is ill-posed if (but not only if) the exponential time dependence is unbounded for $\lambda\in\R$, which poses restrictions on $a$. Avoiding this cause of ill-posedness is equivalent~\cite{FS1999a} to requiring: if $n$ is odd then $a=\pm i$ and if $n$ is even then $\Re(a)\geq0$.

For such a problem to be well-posed, it is necessary and sufficient~\cite{FS1999a} that
\BE \label{eqn:defn.N}
N = \begin{cases}n/2 & n \mbox{ even,} \\ (n+1)/2 & n \mbox{ odd, } a = i, \\ (n-1)/2 & n \mbox{ odd, } a = -i.\end{cases}
\EE
Note that by well-posed, we mean that there exists a unique solution; we make no claims regarding the continuous dependence of the solution on the data.

In the sequel, we develop a spectral theory of the differential operators associated with well-posed IBVP $(S,a)$ of the form~\eqref{eqn:IBVP}.

\subsubsection*{Transform pair}

Let $\alpha = e^{2\pi i/n}$. We define
\BE \label{eqn:M.defn}
\M{M}{k}{j}(\lambda) = \sum_{r=0}^{n-1} (-i\alpha^{k-1}\lambda)^r \Msup{b}{j}{r}{\star}.
\EE
Then the $(n-N)\times(n-N)$ matrix $M(\lambda)$ is an analogue of Birkhoff's adjoint characteristic matrix~\cite{Bir1908b} for the one-point differential operator $S$.

For example,
\BE
M(\lambda) = \BP 1 & -i\lambda \\ 1 & -i\alpha\lambda \EP, \qquad M(\lambda) = \BP1\EP,
\EE
in problems~1 and~2 respectively, the latter being a $1\times1$ matrix.

\begin{defn}
We define 
\begin{itemize}
\item
the polynomial $\Delta (\lambda)$ as the determinant of $M$:
$$\Delta(\lambda) = \det M(\lambda);$$
\item
the $(n-N-1)\times(n-N-1)$ matrix $\Msups{X}{}{}{l}{j}$ as  the submatrix of $M$ with $(1,1)$ entry the $(l+1,j+1)$ entry of the $(2(n-N))\times(2(n-N))$ matrix
\BE
\BP M & M \\ M & M \EP.
\EE
If $N=n-1$, as is the case in problem~2, we adopt the convention that $\Msups{X}{}{}{1}{1}$ is a $0\times0$ matrix with determinant $1$. If $N=n-2$, as is the case in problem~1, then we can simplify
\BE
\det\Msups{X}{}{}{l}{j}(\lambda) = \M{M}{3-l}{3-j}(\lambda).
\EE
\end{itemize}
\end{defn}

We also choose a number $R>0$ such that the open disc $B(0,R)$ contains all zeros of $\Delta$. 

\begin{defn} \label{defn:TransformPair.General}
The transform pair is given by
\begin{subequations} \label{eqn:defn.forward.transform}
\begin{align}
f(x) &\mapsto F[f](\lambda): & F[f](\lambda) &= F_k[f](\lambda), \qquad \lambda\in\Gamma_k,\quad k\in\{0,1,\ldots,N\} \\ \label{eqn:defn.inverse.transform.2}
F(\lambda) &\mapsto f[F](x): & f[F](x) &= \int_{\Gamma} e^{i\lambda x} F(\lambda) \d\lambda, \qquad x\in[0,\infty),
\end{align}
\end{subequations}
where, for $\lambda\in\C$ such that $\Delta(\lambda)\neq0$,
\begin{subequations} \label{eqn:defn.Fpm.rho}
\begin{align}
F_0[f](\lambda) &= \frac{1}{2\pi} \int_0^\infty e^{-i\lambda x} f(x)\d x, \\ \notag
F_k[f](\lambda) &= \frac{1}{2\pi\Delta(\alpha^{N+1-k}\lambda)} \sum_{l=1}^{n-N}\sum_{j=1}^{n-N} (-1)^{(n-N-1)(l+j)} \det \Msups{X}{}{}{l}{j}(\alpha^{N+1-k}\lambda) \\ \label{eqn:defn.Fk}
&\hspace{15ex} \M{M}{1}{j}(\lambda) \int_0^\infty \exp\left(-i\alpha^{N+l-k}\lambda x\right) f(x)\d x,
\end{align}
\end{subequations}
for $k\in\{1,2,\ldots,N\}$ and the contours are defined by
\begin{subequations}
\begin{align}
\Gamma       &= \bigcup_{k=0}^N\Gamma_k, \\
\Gamma_0     &= \begin{matrix}\R \mbox{ perturbed along a semicircular contour of}\\\mbox{ radius $\delta$ above } 0, \mbox{ with \emph{positive} orientation,}\end{matrix} \\
\delta &> 0      \mbox{  arbitrarily small, and} \\
\Gamma_k     &= \begin{matrix}\mbox{the $k$\Th connected component of } \partial(\{\lambda\in\C^+:\Re(a\lambda^n)<0,|\lambda|>R\}), \\ \mbox{ counting anticlockwise from } \R^+, \mbox{ with \emph{negative} orientation,}\end{matrix}
\end{align}
\end{subequations}
for $k\in\{1,2,\ldots,N\}$.
\end{defn}

For problem~1, $\Delta(\alpha\lambda)=-i\lambda(\alpha^2-\alpha)$, and equation~\eqref{eqn:defn.Fk} simplifies to the original definition of $F_1$, equations~\eqref{eqn:introTrans.1.1a} and~\eqref{eqn:introTrans.1.1c}. For problem~2, $\Delta(\alpha^2\lambda)=\Delta(\alpha\lambda)=1$ and equation~\eqref{eqn:defn.Fk} immediately simplifies to give the expected definitions of $F_1$ and $F_2$.

\begin{figure}
\begin{center}
\includegraphics{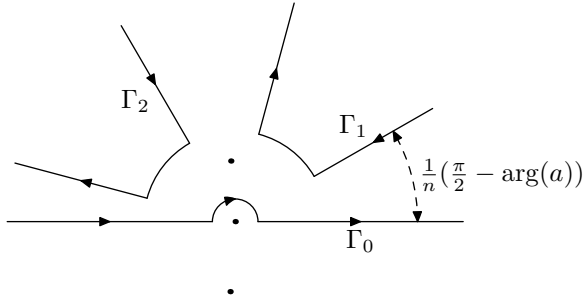}
\caption{Definition of the contour $\Gamma$.}
\label{fig:general-contdef}
\end{center}
\end{figure}

The contours $\Gamma_j$ for problems~1 and~2 are shown on figure~\ref{fig:3o-cont}. Figure~\ref{fig:general-contdef} shows the position of the contours for the problem $(S,e^{-i\frac{\pi}{6}})$, where $n=4$ and the boundary forms are
\BE
B_1\phi = \phi'''(0)+3\phi''(0), \qquad B_2\phi = \phi'(0)-2\phi(0).
\EE
As the boundary conditions are of Robin type, the characteristic determinant $\Delta$ has nonzero zeros. Indeed, $\Delta(\lambda)=0$ at the dots in figure~\ref{fig:general-contdef} (there is a double zero at zero) but $R=4$ is sufficient to ensure $\Delta\neq0$ outside the disc $B(0,R)$.

\subsection{General case: validity of transform pair} \label{ssec:GeneralValidity}

\begin{prop} \label{prop:Transforms.valid:General:2}
Let $S$ be an operator corresponding to a well-posed initial-boundary value problem and let $(F[\cdot],f[\cdot])$ be the transform pair given by definition~\ref{defn:TransformPair.General}. Then for all $f\in\Phi$ and for all $x\in(0,\infty)$,
\BE \label{eqn:Transforms.valid:General:prop2}
f[F[f]](x) = \sum_{k=0}^N\int_{\Gamma_k} e^{i\lambda x} F_k[f](\lambda) \d\lambda = f(x).
\EE
\end{prop}

This is a direct generalisation of proposition~\ref{prop:Transforms.valid:LKdV:2}; its proof follows that of the earlier proposition with an equivalent application of Jordan's lemma.

\begin{proof}
For $k\in\{1,2,\ldots,N\}$, we consider the integral
\begin{multline} \label{eqn:Transforms.valid:General:2.proof:1}
\int_{\Gamma_k} e^{i\lambda x} F_k[f](\lambda) \d\lambda = \frac{1}{2 \pi}  \int_{\Gamma_k} e^{i\lambda x} \sum_{l=1}^{n-N} \int_0^\infty \exp\left(-i\alpha^{N+l-k}\lambda y\right) f(y)\d y \\
\left [\sum_{j=1}^{n-N}(-1)^{(n-N-1)(l+j)} \frac{\det \Msups{X}{}{}{l}{j}(\alpha^{N+1-k}\lambda) \M{M}{1}{j}(\lambda)}{\Delta(\alpha^{N+1-k}\lambda)} \right] \d\lambda.
\end{multline}
The square bracket represents a meromorphic function, which is holomorphic and bounded on $\Gamma_k$ and the region lying to the right of $\Gamma_k$. The inner integral is entire and decaying like $O(\lambda^{-1})$ as $\lambda\to\infty$ along $\Gamma_k$ in either direction or from within the sector to the right of $\Gamma_k$. Hence, by Jordan's lemma, integral~\eqref{eqn:Transforms.valid:General:2.proof:1} evaluates to $0$. Note: we are `closing' the contour $\Gamma_k$ by moving it to the right; see figure~\ref{fig:general-contdef}.

The integrand $e^{i\lambda x}F_0[f](\lambda)$ is holomorphic hence $\Gamma_0$ may be deformed onto $\R$ and, by the validity of the usual Fourier transform, the transform pair is valid:
\BE
f[F[f]](x) = \sum_{k=0}^N\int_{\Gamma_k} e^{i\lambda x} F_k[f](\lambda)\d\lambda = \int_{\R} e^{i\lambda x}F_0[f](\lambda)\d \lambda = f(x).
\EE
\end{proof}

\section{Fokas' unified transform method for IBVP} \label{sec:Transform.Method}

In section~\ref{ssec:Transform.Method:LKdV} we prove equation~\eqref{eqn:introIBVP.solution.transform.1} for the transform pairs~\eqref{eqn:introTrans.1.1}. In section~\ref{ssec:Transform.Method:LKdV:General}, we establish equivalent results for general well-posed initial-boundary value problems.

\subsection{Linearized KdV} \label{ssec:Transform.Method:LKdV}

\begin{prop} \label{prop:Transform.Method:LKdV:2}
The solution of problem~1 is given by equation~\eqref{eqn:introIBVP.solution.transform.1}, with $F[f](\lambda)$ and $f[F](x)$ defined by equations~\emph{\eqref{eqn:introTrans.1.1a}--\eqref{eqn:introTrans.1.1c}}.

The solution of problem~2 is given by equation~\eqref{eqn:introIBVP.solution.transform.1}, with $F[f](\lambda)$ and $f[F](x)$ defined by equations~\eqref{eqn:introTrans.1.1a}--\eqref{eqn:introTrans.1.1bb},~\eqref{eqn:introTrans.1.1d} and~\eqref{eqn:introTrans.1.1e}.
\end{prop}

\begin{proof}
We present the proof for problem~1. The proof for problem~2 is very similar.

Suppose $q(x,t)$, for which $t\mapsto q(\cdot,t)$ is a $C^\infty$ map from $[0,T]$ into $C$, is a solution of the problem~\eqref{eqn:introIBVP.2}. Applying the forward transform to $q$ yields
\BE
F[q(\cdot,t)](\lambda) = \begin{cases} \int_0^\infty \phi^1(x,\lambda)q(x,t)\d x & \mbox{if } \lambda\in\Gamma_1, \\ \int_0^\infty \phi^0(x,\lambda)q(x,t)\d x & \mbox{if } \lambda\in\Gamma_0, \end{cases}
\EE
where $\phi^1$, $\phi^0$ are given by equations~\eqref{eqn:introTrans.1.1c} and~\eqref{eqn:introTrans.1.1bb}. The PDE and integration by parts imply
\begin{align*}
\frac{\d}{\d t} F[q(\cdot,t)](\lambda) &= \int_0^\infty \phi^k(x,\lambda)q_{xxx}(x,t)\d x \\
&= \partial_{x}^2q(0,t) \phi^k(0,\lambda) - \partial_{x}q(0,t) \partial_{x}\phi^k(0,\lambda) \\
&\hspace{10em} + q(0,t) \partial_{xx}\phi^k(0,\lambda) + i\lambda^3 F[q(\cdot,t)](\lambda).
\end{align*}
Rearranging, multiplying by $e^{-i\lambda^3t}$, integrating over $t$ and applying the initial condition, we find
\BE
F[q(\cdot,t)](\lambda) = e^{i\lambda^3t}F[f](\lambda) + e^{i\lambda^3t} \sum_{j=0}^2 (-1)^j \partial_{x}^{2-j}\phi^k(0,\lambda) Q_j(0,\lambda),
\EE
where
\BE
Q_j(x,\lambda) = \int_0^t e^{-i\lambda^3s} \partial_x^j q(x,s) \d s.
\EE
Evaluating $\partial_{x}^{j}\phi^k(0,\lambda)$, we obtain
\begin{multline}
F[q(\cdot,t)](\lambda) = e^{i\lambda^3t}F[f](\lambda) + \frac{e^{i\lambda^3t}}{2\pi} \left[ - Q_0(0,\lambda)2\lambda^2 + Q_1(0,\lambda)(\alpha+\alpha^2) i\lambda \right. \\
\left. + Q_2(0,\lambda)(\alpha+\alpha^2) \right],
\end{multline}
for all $\lambda\in\Gamma_1$ and
\BE
F[q(\cdot,t)](\lambda) = e^{i\lambda^3t}F[f](\lambda) + \frac{e^{i\lambda^3t}}{2\pi} \left[ - Q_0(0,\lambda)\lambda^2 + Q_1(0,\lambda) i\lambda + Q_2(0,\lambda) \right],
\EE
for all $\lambda\in\Gamma_0$.

Hence, the validity of the transform pair, proposition~\ref{prop:Transforms.valid:LKdV:2}, implies
\begin{multline} \label{eqn:Transform.Method:LKdV.2:q.big}
q(x,t) = \left\{\int_{\Gamma_1}+\int_{\Gamma_0}\right\} e^{i\lambda x+i\lambda^3t} F[f](\lambda)\d\lambda \\
+\frac{1}{2\pi} \int_{\Gamma_1} e^{i\lambda x+i\lambda^3t} \left[ - Q_0(0,\lambda)2\lambda^2 \right] \d\lambda + \frac{1}{2\pi} \int_{\Gamma_0} e^{i\lambda x+i\lambda^3t} \left[ - Q_0(0,\lambda)\lambda^2 \right] \d\lambda \\
+ \frac{-\alpha-\alpha^2}{2\pi} \int_{\Gamma_1} e^{i\lambda x+i\lambda^3t} \left[ Q_2(0,\lambda) + Q_1(0,\lambda) i\lambda \right] \d\lambda \\
+ \frac{1}{2\pi} \int_{\Gamma_0} e^{i\lambda x+i\lambda^3t} \left[ Q_2(0,\lambda) + Q_1(0,\lambda) i\lambda \right] \d\lambda.
\end{multline}
Integration by parts yields
\BE
Q_j(x,\lambda) = O(\lambda^{-3}),
\EE
as $\lambda\to\infty$ within the sectors $0\leq\arg\lambda\leq\pi/3$ and $2\pi/3\leq\arg\lambda\leq\pi$. Further, the integrands on the third and fourth lines of equation~\eqref{eqn:Transform.Method:LKdV.2:q.big} are entire. Hence, by Jordan's lemma (used to `open' the contour $\Gamma_1$ to the left until it coincides with $\Gamma_0$ but with opposite orientation) and noting $-\alpha-\alpha^2=1$, the integrands on the third and fourth lines cancel. The boundary conditions imply
\BE
Q_0(0,\lambda) = 0,
\EE
so the second line of equation~\eqref{eqn:Transform.Method:LKdV.2:q.big} vanishes. Hence
\BE
q(x,t) = \left\{\int_{\Gamma_1}+\int_{\Gamma_0}\right\} e^{i\lambda x+i\lambda^3t} F[f](\lambda)\d\lambda.
\EE
\end{proof}

The above proof also demonstrates how the transform pair may be used to solve a problem with inhomogeneous boundary conditions: consider the problem
\begin{subequations} \label{eqn:introIBVP.2inhomo}
\begin{align} \label{eqn:introIBVP.2inhomo:PDE}
q_t(x,t) + q_{xxx}(x,t) &= 0 & (x,t) &\in (0,\infty)\times(0,T), \\
q(x,0) &= \phi(x) & x &\in [0,\infty), \\
q(0,t) &= h(t) & t &\in [0,T],
\end{align}
\end{subequations}
for some given Dirichlet boundary datum $h\in C^\infty[0,T]$ compatible with $f$. Then $Q_0(0,\lambda)$ is nonzero but is a known quantity, namely the $t$-transform of the boundary datum. Substituting this value into equation~\eqref{eqn:Transform.Method:LKdV.2:q.big} yields an explicit expression for the solution.

\subsection{General case} \label{ssec:Transform.Method:LKdV:General}

\begin{prop} \label{prop:Transform.Method:General:2}
The solution of a well-posed initial-boundary value problem is given by
\BE \label{eqn:Transform.Method:General:prop2:1}
q(x,t) = f\left[ e^{-a\lambda^nt} F[f] \right](x),
\EE
where $(F[\cdot],f[\cdot])$ is the transform pair of definition~\ref{defn:TransformPair.General}.
\end{prop}

The principal tool in the proof of proposition~\ref{prop:Transform.Method:General:2} is the following lemma.

\begin{lem} \label{lem:GEInt1}
Let $f\in\Phi$ and $S$ be the differential operator defined in equation~\eqref{eqn:defnS}. Then there exists a polynomial $P_f$ of degree at most $n-1$ such that, for all $k\in\{0,1,\ldots,N\}$,
\BE
F_k[Sf](\lambda) = \lambda^n F_k[f](\lambda) + P_f(\lambda),
\EE
with $P_f$ independent of $k$.
\end{lem}

If it held that $P_f=0$ then this lemma would simply state that each $F_k[\cdot](\lambda)$ was a generalised eigenfunction of $S$ and proposition~\ref{prop:Transform.Method:General:2} would follow by proposition~\ref{prop:Transforms.valid:General:2}. Although lemma~\ref{lem:GEInt1} is weaker than $P_f=0$, the $o(\lambda^n)$ as $\lambda\to\infty$ bound on $P_f$ is sufficient to give proposition~\ref{prop:Transform.Method:General:2}. That is, we will be able to show that $f [ e^{-a\lambda^nt} P_f ] (x) = 0$.

\begin{proof}[Proof of Lemma~\ref{lem:GEInt1}]
Let $(f,g)$ be the usual inner product $\int_0^\infty f(x)\bar{g}(x)\d x$. For any $\lambda\in\Gamma$, we can represent $F_k[\cdot](\lambda)$ as the inner product $F_k[f](\lambda)=(f,\phi^k_\lambda)$, for the function $\phi^k_\lambda(x)$, rational in $\lambda$ and smooth in $x$, defined by
\begin{subequations} \label{eqn:defn.fpm.rho}
\begin{align}
\overline{\phi^0_\lambda}(x) &= \frac{1}{2\pi} e^{-i\lambda x}, \\ \notag
\overline{\phi^k_\lambda}(x) &= \frac{1}{2\pi\Delta(\alpha^{N+1-k}\lambda)} \sum_{l=1}^{n-N}\sum_{j=1}^{n-N} (-1)^{(n-N-1)(l+j)} \det \Msups{X}{}{}{l}{j}(\alpha^{N+1-k}\lambda) \\
&\hspace{31ex} \M{M}{1}{j}(\lambda) \exp(-i\alpha^{N+l-k}\lambda x),
\end{align}
\end{subequations}
for $k\in\{1,2,\ldots,N\}$.

If $\lambda\in\Gamma_k$, then $\phi^k_\lambda$ is a smooth and bounded function of $x$. Also, $Sf\in C$ and $\alpha^{(l-1)n}=1$, so equation~\eqref{eqn:S.not.s-a} yields
\BE
F_k[Sf](\lambda) = \lambda^n F_k[f](\lambda) - [f \phi^k_\lambda](0).
\EE
If $B$, $B^\star:C\to\C^n$, are the vector boundary forms
\BE
B=(B_1,B_2,\ldots,B_N), \qquad B^\star=(B_1^\star,B_2^\star,\ldots,B_{n-N}^\star),
\EE
then there exist complimentary vector boundary forms $B_c$, $B_c^\star$ such that
\BE \label{eqn:propGEInt.3}
- [f \phi^k_\lambda](0) = Bf \cdot B_c^\star \phi^k_\lambda + B_cf \cdot B^\star \phi^k_\lambda,
\EE
where $\cdot$ is the usual sesquilinear dot product of vectors. This follows by considering the finite-interval case~\cite[chapter~11]{CL1955a} and taking the limit $($length of interval$)\to\infty$ and imposing compact support (rapid decay is sufficient). We consider the right hand side of equation~\eqref{eqn:propGEInt.3} as a function of $\lambda$. As $Bf=0$, this expression is a linear combination of the functions $\overline{B_r^\star\phi^k_\lambda}$ of $\lambda$, with coefficients given by the complementary boundary forms.

The definitions of $\overline{B_r^\star}$ and $\phi_\lambda^k$ imply, for $k\geq1$,
\begin{align} \notag
\overline{B_r^\star\phi^k_\lambda} &= \frac{1}{2\pi\Delta(\alpha^{N+1-k}\lambda)} \sum_{l=1}^{n-N}\sum_{j=1}^{n-N} (-1)^{(n-N-1)(l+j)} \det \Msups{X}{}{}{l}{j}(\alpha^{N+1-k}\lambda) \\
&\hspace{25ex} \M{M}{1}{j}(\lambda) \overline{B_r^\star}\left(\exp\left(-i\alpha^{N+l-k}\lambda \; \cdot \; \right)\right) \\ \notag
&= \frac{1}{2\pi\Delta(\alpha^{N+1-k}\lambda)} \sum_{l=1}^{n-N}\sum_{j=1}^{n-N} (-1)^{(n-N-1)(l+j)} \det \Msups{X}{}{}{l}{j}(\alpha^{N+1-k}\lambda) \\
&\hspace{25ex} \M{M}{1}{j}(\lambda) \M{M}{l}{r}(\alpha^{N+1-k}\lambda).
\end{align}
But
\BE
\sum_{l=1}^{n-N} (-1)^{(n-N-1)(l+j)} \det\Msups{X}{}{}{l}{j}(\alpha^{N+1-k}\lambda)\M{M}{l}{r}(\alpha^{N+1-k}\lambda) = \Delta(\alpha^{N+1-k}\lambda)\M{\delta}{j}{r},
\EE
so
\begin{align}
\overline{B_r^\star \phi^k_\lambda} &= \frac{1}{2\pi}\M{M}{1}{r}(\lambda). \\
\intertext{By definition,} \label{eqn:Pf.k0}
\overline{B_r^\star \phi^0_\lambda} &= \frac{1}{2\pi}\M{M}{1}{r}(\lambda).
\end{align}
Finally, by equations~\eqref{eqn:M.defn}, $\M{M}{1}{r}$ is a polynomial of order at most $n-1$.
\end{proof}

\begin{proof}[Proof of Proposition~\ref{prop:Transform.Method:General:2}]
Let $q$ be the solution of the initial-boundary value problem. Then, since $q$ satisfies the partial differential equation~\eqref{eqn:IBVP.PDE}, for each $k\in\{0,1,\ldots,N\}$,
\BE
\frac{\d}{\d t} F_k[q(\cdot,t)](\lambda) = -aF_k[S(q(\cdot,t))](\lambda) = -a\lambda^n F_k[q(\cdot,t)](\lambda) - a P_{q(\cdot,t)}(\lambda),
\EE
where, by lemma~\ref{lem:GEInt1}, $P_{q(\cdot,t)}$ is a polynomial of degree at most $n-1$ independent of $k$. Hence
\BE
\frac{\d}{\d t} \left( e^{a\lambda^nt}F_k[q(\cdot,t)](\lambda) \right) = -ae^{a\lambda^nt}P_{q(\cdot,t)}(\lambda).
\EE
Integrating with respect to $t$ and applying the initial condition~\eqref{eqn:IBVP.IC}, we find
\BE \label{eqn:Transform.Method:General:1proof.1}
F_k[q(\cdot,t)](\lambda) = e^{-a\lambda^nt}F_k[f](\lambda) -a e^{-a\lambda^nt} \int_0^t e^{a\lambda^ns} P_{q(\cdot,s)}(\lambda)\d s.
\EE
The validity of the transform pair, proposition~\ref{prop:Transforms.valid:General:2}, implies
\BE \label{eqn:Transform.Method:General:2proof.3}
q(x,t) = \sum_{k=0}^N\int_{\Gamma_k} e^{i\lambda x-a\lambda^nt}F_k[f](\lambda) \d\lambda
- a \sum_{k=0}^N\int_{\Gamma_k} e^{i\lambda x-a\lambda^nt} \left(\int_0^t e^{a\lambda^ns} P_{q(\cdot,s)}(\lambda) \d s\right) \d\lambda.
\EE
If $t=0$, the latter integrand is $0$ and the result holds. Otherwise, the latter integrand is entire and integration by parts yields
\BE
e^{i\lambda x-a\lambda^nt} \left(\int_0^t e^{a\lambda^ns} P_{q(\cdot,s)}(\lambda) \d s\right) = O(\lambda^{-1}) \mbox{ as } \lambda\to\infty
\EE
within the closed sectors $\{\lambda\in\C^+:\Re(a\lambda^n)\geq0\}$.
Hence, by Jordan's lemma applied to the set
\BE
\C^+\setminus\bigcup_{k=1}^N\{\lambda \mbox{ lying to the left of } \Gamma_k\},
\EE
the latter integral of equation~\eqref{eqn:Transform.Method:General:2proof.3} vanishes.

Indeed, for each $k\in\{1,2,\ldots,N\}$, we use Jordan's lemma to `open' each contour $\Gamma_k$ by rotating the semi-infinite components about $0$ until they coincide with semi-infinite components of $\Gamma_{k+1}$, $\Gamma_{k-1}$ or $\Gamma_0$ with opposite orientation but the same integrand. Thus the contributions of the semi-infinite components of $\Gamma_k$ mutually annihilate and we are left with
\BE
\int_\gamma e^{i\lambda x-a\lambda^nt} \left(\int_0^t e^{a\lambda^ns} P_{q(\cdot,s)}(\lambda) \d s\right) \d\lambda,
\EE
where $\gamma$ is the contour
\BE
\partial \left[ \left( D(0,R) \setminus D(0,\delta) \right) \cap \C^+ \right],
\EE
with positive orientation. The integrand is entire, so the integral vanishes.
\end{proof}

\section{Analysis of the transform pair} \label{sec:Spectral}

In this section we analyse the spectral properties of the transform pairs using the notion of augmented eigenfunctions.

\subsection{Linearized KdV} \label{ssec:Spectral:LKdV}
The main results in this section are the following.

\begin{thm} \label{thm:Diag:LKdV:1}
The transform pairs $(F[\cdot],f[\cdot])$ defined by equations~\emph{\eqref{eqn:introTrans.1.1a}--\eqref{eqn:introTrans.1.1c}} and defined by equations~\emph{\eqref{eqn:introTrans.1.1a}--\eqref{eqn:introTrans.1.1bb}},~\eqref{eqn:introTrans.1.1d} and~\eqref{eqn:introTrans.1.1e} provide spectral representations of the spatial differential operators associated with problems~1 and~2, respectively, in the sense of definition~\ref{defn:Spect.Rep.II}.
\end{thm}

\begin{thm} \label{thm:Diag:LKdV:2}
The transform pair $(F[\cdot],f[\cdot])$ defined in~\emph{\eqref{eqn:introTrans.1.1a}--\eqref{eqn:introTrans.1.1c}} provides a spectral representation of the spatial differential operator associated with problem~1 in the sense of definition~\ref{defn:Spect.Rep.I.II}.
\end{thm}

\subsubsection*{Augmented Eigenfunctions}

Let $S^{1}$ and $S^{2}$ be the differential operators representing the spatial parts of the IBVPs~1 and~2, respectively. Each operator is a restriction of the same formal differential operator, $(-i\d/\d x)^3$ to the domain of initial data compatible with the boundary conditions of the problem:
\begin{align} \label{eqn:introS1}
\mathcal{D}(S^{1}) &= \{f\in C: f(0)=0\}, \\ \label{eqn:introS2}
\mathcal{D}(S^{2}) &= \{f\in C: f(0)=f'(0)=0\}.
\end{align}

Integration by parts yields
\begin{align} \label{eqn:introS1.Rphiplus}
F_1[S^{1}f](\lambda) &= \lambda^3 F_1[f](\lambda) + \left( \frac{i}{2\pi}f''(0) - \frac{\lambda}{2\pi}f'(0) \right), \\  \label{eqn:introS1.Rphiminus}
F_0[S^{1}f](\lambda) &= \lambda^3 F_0[f](\lambda) + \left( - \frac{i}{2\pi}f''(0) + \frac{\lambda}{2\pi}f'(0) \right). \\
\intertext{Similarly,} \label{eqn:introS2.Rphiplus}
F_k[S^{2}f](\lambda) &= \lambda^3 F_k[f] + \left( \frac{i}{2\pi}f''(0) \right), \qquad k\in\{1,2\}, \\  \label{eqn:introS2.Rphiminus}
F_0[S^{2}f](\lambda) &= \lambda^3 F_0[f] + \left( - \frac{i}{2\pi}f''(0) \right).
\end{align}
In each case, the remainder functional, which is enclosed in parentheses, is entire in $\lambda$.

The ratios of the remainder functionals to the eigenvalue are rational functions with no pole in the regions to the right of $\Gamma_k$ and decaying as $\lambda\to\infty$. Jordan's lemma applied in the sectors to the right of $\Gamma_k$, $k\geq1$ and in the upper half-plane for $\Gamma_0$ implies~\eqref{eqn:defnAugEig.Control2} hence $\{F_\lambda:\lambda\in\bigcup_{k\geq0}\Gamma_k\}$ is a family of type~II augmented eigenfunctions of the corresponding $S^{1}$ or $S^{2}$.

\begin{rmk} \label{rmk:Gamma0.shape}
Suppose that we wished to work in in $\mathcal{S}[0,\infty)$ directly, instead of the space of compactly-supported functions. Then, in order to establish the validity (or even the definition) of the transform pair one must insist $\Gamma_0\cap\C^+=\emptyset$. It is now clear why we avoid this approach, and choose to deform $\Gamma_0$ away from $0$ and into $\C^+$, not $\C^-$. Indeed, otherwise, applying Jordan's lemma to
\BE
\int_{\Gamma_0}e^{i\lambda x}\frac{1}{\lambda^3}\left( \frac{-i}{2\pi}f''(0) \right)\d \lambda,
\EE
we would pick up a contribution from the pole at zero, hence $\{F[\cdot](\lambda):\lambda\in\Gamma_0\}$ would fail to be a family of type~II augmented eigenfunctions.
\end{rmk}

\subsubsection*{Spectral representation of $S^{2}$ - proof of Theorem \ref{ssec:Spectral:LKdV}}

We have shown above that $\{F[\cdot](\lambda):\lambda\in\bigcup_{k\geq0}\Gamma_k\}$ is a family of type~II augmented eigenfunctions of $S^{2}$ with eigenvalue $\lambda^3$. Moreover, by proposition~\ref{prop:Transforms.valid:LKdV:2},
\BE
\int_{\Gamma_0\cup\Gamma_1\cup\Gamma_2}e^{i\lambda x}F[f](\lambda)\d\lambda
\EE
converges to $f$. This completes the proof of theorem~\ref{thm:Diag:LKdV:1} for problem~2.

\subsubsection*{Spectral representation of $S^{1}$ - proof of Theorem \ref{thm:Diag:LKdV:2}}

By the above argument, it is clear that the transform pair $(F[\cdot],f[\cdot])$ defined by equations~\eqref{eqn:introTrans.1.1a}--\eqref{eqn:introTrans.1.1c} provides a spectral representation of $S^{1}$ in the sense of definition~\ref{defn:Spect.Rep.II}, establishing theorem~\ref{thm:Diag:LKdV:1} for problem~1.

For $x\in(0,\infty)$,
\BE
\int_{\Gamma_1} e^{i\lambda x} \left(\frac{i}{2\pi} f''(0) - \frac{\lambda}{2\pi} f'(0) \right) \d \lambda = \int_{\Gamma_1} e^{i\lambda x/2} \left[ \frac{e^{i\lambda x/2}}{2\pi} \left(i f''(0) - \lambda f'(0) \right) \right] \d \lambda,
\EE
the integrand is entire and the square bracket decays exponentially as $\lambda\to\infty$ from within the closed sector $\frac{\pi}{3}\leq\arg\lambda\leq\frac{2\pi}{3}$. Hence, by Jordan's lemma, the integral converges to $0$. We have shown that $\{F[\cdot](\lambda):\lambda\in\Gamma_1\}$ is a family of type~I augmented eigenfunctions of $S^{1}$.

Note that this holds precisely because, as $\lambda\to\infty$ along any semi-infinite component of $\Gamma_1$, $\Im(\lambda)\to +\infty$. In particular, $\{F[\cdot](\lambda):\lambda\in\Gamma_0\}$ is \emph{not} a family of type~I augmented eigenfunctions of $S^1$. It is clear that $\{F[\cdot](\lambda):\lambda\in\Gamma_0\cup\Gamma_1\}$ is not a family of type~I augmented eigenfunctions, so we cannot provide a spectral representation of $S^1$ using only type~I augmented eigenfunctions. Similarly, neither $\{F[\cdot](\lambda):\lambda\in\Gamma_1\}$ nor $\{F[\cdot](\lambda):\lambda\in\Gamma_2\}$ is a family of type~I augmented eigenfunctions of $S^2$.

Convergence of
\BE
\int_{\Gamma_1}e^{i\lambda x}F[S^1f](\lambda)\d\lambda
\EE
for all $f\in C$ with $f(0)=0$ follows by the argument in the proof of proposition~\ref{prop:Transforms.valid:LKdV:2}, except in this case $S^1f\in C$ but not necessarily $(S^1f)(0)=0$ so we can only guarantee $\hat{f}(\alpha\lambda)$, $\hat{f}(\alpha^2\lambda) = O(\lambda^{-1})$.

This completes the proof of theorem~\ref{thm:Diag:LKdV:2}.

\subsection{General case} \label{ssec:Spectral:General}

We will show that the transform pair $(F[\cdot],f[\cdot])$ represents spectral decomposition by type~II augmented eigenfunctions.

\begin{thm} \label{thm:Diag.2}
Let $S$ be the spatial differential operator associated with a well-posed IBVP. Then the transform pair $(F[\cdot],f[\cdot])$ given by definition~\ref{defn:TransformPair.General} provides a spectral representation of $S$ in the sense of definition~\ref{defn:Spect.Rep.II}.
\end{thm}

Let $(S,a)$ be such that the associated initial-boundary value problem is well-posed. Then there exists a complete system of augmented eigenfunctions associated with $S$. The augmented eigenfunctions are all of type~\textup{II}. However, in certain cases, some of the augmented eigenfunctions are also of type~\textup{I}.

\begin{prop} \label{prop:GEIntComplete3}
For each $k\in\{0,1,2,\ldots,N\}$, we define the system of functionals
\BE
\mathcal{F}_k = \{F_k[\cdot](\lambda):\lambda\in\Gamma_k\}.
\EE
\begin{enumerate}
\item[(i)]{For each $k\in\{0,1,2,\ldots,N\}$, $\mathcal{F}_k$ is a family of type~\textup{II} augmented eigenfunctions of $S$ up to integration over $\Gamma_k$, with eigenvalues $\lambda^n$.}
\item[(ii)]{If either $n$ is odd and $a=-i$ or $n$ is even and $\Re(a)>0$, then, for each $k\in\{1,2,\ldots,N\}$, $\mathcal{F}_k$ is a family of type~\textup{I} augmented eigenfunctions of $S$ up to integration over $\Gamma_k$, with eigenvalues $\lambda^n$.}
\item[(iii)]{If the initial-boundary value problem $(S,a)$ is well-posed, then $\mathcal{F}=\bigcup_{k=0}^N \mathcal{F}_k$ is a complete system.}
\end{enumerate}
\end{prop}

\begin{proof} ~
\begin{enumerate}
\item[(i)]{For each $k\in\{1,\ldots,N\}$, lemma~\ref{lem:GEInt1} implies
\BE \label{eqn:GEIntComplete3:Proof:1}
\int_{\Gamma_k} e^{i\lambda x}\lambda^{-n}(F_k[Sf](\lambda) - \lambda^nF_k[f](\lambda))\d\lambda = \int_{\Gamma_k} e^{i\lambda x}\lambda^{-n}P_f(\lambda)\d\lambda,
\EE
and the integrand is the product of $e^{i\lambda x}$ with an entire function decaying as $\lambda\to\infty$. Hence, by Jordan's lemma applied on the region to the right of $\Gamma_k$, the integral of the remainder functionals vanishes for all $x>0$. Equation~\eqref{eqn:GEIntComplete3:Proof:1} also holds for $k=0$ and we can apply Jordan's lemma on $\C^+$.}
\item[(ii)]{If $(n,a)$ obey the specified conditions, then, for $k\in\{1,2,\ldots,N\}$, $\Gamma_k$ is disjoint from $\R$ and $\Im(\lambda)\to+\infty$ as $\lambda\to\infty$ along either semi-infinite component of $\Gamma_k$. By lemma~\ref{lem:GEInt1},
\BE
\int_{\Gamma_k} e^{i\lambda x}(F_k[Sf](\lambda) - \lambda^nF_k[f](\lambda))\d\lambda = \int_{\Gamma_k} e^{i\lambda x/2}\left(e^{i\lambda x/2}P_f(\lambda)\right)\d\lambda,
\EE
and the integrand is the product of $e^{i\lambda x/2}$ with a function analytic on the enclosed set and decaying as $\lambda\to\infty$. Hence, by Jordan's lemma, the integral of the remainder functionals vanishes for all $x>0$.}
\item[(iii)]{Considering $f\in\Phi$ as the initial datum of the homogeneous initial-boundary value problem and applying proposition~\ref{prop:Transform.Method:General:2}, we evaluate the solution of problem~\eqref{eqn:IBVP} at $t=0$,
\BE
f(x) = q(x,0) = \sum_{k=0}^N\int_{\Gamma_k} e^{i\lambda x} F_k[f](\lambda) \d\lambda.
\EE
Thus if for all $k\in\{0,1,\ldots,N\}$ and for all $\lambda\in\Gamma_k$, $F_k[f](\lambda)=0$, then $f=0$.}
\end{enumerate}
\end{proof}

\begin{proof}[Proof of Theorem~\ref{thm:Diag.2}]
Proposition~\ref{prop:GEIntComplete3} establishes completeness of the augmented eigenfunctions and equation~\eqref{eqn:Spect.Rep.I}, under the assumption that the integrals converge. Theorem~\ref{prop:Transforms.valid:General:2} implies the required convergence.
\end{proof}

\begin{thm} \label{thm:Diag.5}
Let $(S,a)$ be a well-posed IBVP such that either $n$ is odd and $a=-i$ or $n$ is even and $\Re(a)>0$. Then the transform pair $(F[\cdot],f[\cdot])$ given by definition~\ref{defn:TransformPair.General} provides a spectral representation of $S$ in the sense of definition~\ref{defn:Spect.Rep.I.II}.
\end{thm}

\begin{proof}
Proposition~\ref{prop:GEIntComplete3} establishes that $\bigcup_{k=1}^N \mathcal{F}_k$ is a family of type~I augmented eigenfunctions up to integration over $\bigcup_{k=1}^N \Gamma_k$, that $\mathcal{F}_0$ is a family of type~II augmented eigenfunctions up to integration over $\Gamma_0$ and that $\mathcal{F}$ is complete. It only remains to establish convergence of
\BE \label{eqn:Diag.5:Proof:1}
\int_{\Gamma_k} e^{i\lambda x} F_k[Sf](\lambda)\d\lambda,
\EE
for all $x\in(0,\infty)$, all $k\in\{1,2,\ldots,N\}$ and all $f\in\Phi$. By lemma~\ref{lem:GEInt1}, the integral~\eqref{eqn:Diag.5:Proof:1} may be written
\BE
\int_{\Gamma_k} e^{i\lambda x/2} \left[ e^{i\lambda x/2} \left( \lambda^n F_k[f](\lambda) + P_f(\lambda) \right) \right] \d\lambda,
\EE
where $F_k[f](\lambda)$ is bounded and holomorphic on $\Gamma_k$ and the region lying to the right of $\Gamma_k$, and $P_f$ is a polynomial. Hence, by Jordan's lemma, this integral converges to $0$.
\end{proof}

\begin{rmk} \label{rmk:Gamma.k.ef.at.inf}
As our choice of $R$ in the definition of $\Gamma_k$, for $k\in\{1,2,\ldots,N\}$, may be arbitrarily large, the contours $\Gamma_k$ need not pass through any finite region. By considering the limit $R\to\infty$, we claim that $\bigcup_{k=1}^N\mathcal{F}_k$ can be seen to represent \emph{spectral objects with eigenvalue at infinity}.
\end{rmk}

\section{Conclusions}
In this paper we have elucidated the spectral meaning of the integral representation for the solution of a well-posed half-line IBVP, given by the unified transform method of Fokas. We proved that this approach  can be used to construct a transform-inverse transform pair tailored to the problem, where the forward transform can be viewed as a family of type~II augmented eigenfunctions. Moreover, these type~II augmented eigenfunctions provide a spectral representation of the associated differential operator in the sense of definition~\ref{defn:Spect.Rep.II}

The definition of augmented eigenfunctions is a direct extension of the ``generalised eigenfunctions'' introduced by Gelfand, and has clear analogies with the ``pseudo eigenfunctions'' as described e.g.\ in \cite{ET2005a}. The crucial difference is that the augmented eigenfunctions are only defined modulo terms that are analytic in certain subdomains of $\C$, and the appropriate use of analyticity and Cauchy's theorem are crucial for our results.

\subsubsection*{Acknowledgement}
The research leading to these results has received funding from the European Union's Seventh Framework Programme FP7-REGPOT-2009-1 under grant agreement n$^\circ$ 245749.

\bibliographystyle{amsplain}
\bibliography{dbrefs}

\providecommand{\bysame}{\leavevmode\hbox to3em{\hrulefill}\thinspace}
\providecommand{\MR}{\relax\ifhmode\unskip\space\fi MR }
\providecommand{\MRhref}[2]{%
  \href{http://www.ams.org/mathscinet-getitem?mr=#1}{#2}
}
\providecommand{\href}[2]{#2}
\begin{thebibliography}{10}

\bibitem{Bir1908b}
G.~D. Birkhoff, \emph{Boundary value and expansion problems of ordinary linear
  differential equations}, Trans. Amer. Math. Soc. \textbf{9} (1908), 373--395.

\bibitem{CL1955a}
E.~A. Coddington and N.~Levinson, \emph{Theory of ordinary differential
  equations}, International Series in Pure and Applied Mathematics,
  McGraw-Hill, 1955.

\bibitem{DV2011a}
B.~Deconinck and V.~Vasan, \emph{Well-posedness of boundary-value problems for
  the linear {B}enjamin-{B}ona-{M}ahony equation}, Discrete Contin. Dyn. Syst.
  \textbf{33} (2013), no.~7, 3171--3188.

\bibitem{ET2005a}
M.~Embree and L.~N. Trefethen, \emph{Spectra and pseudospectra}, Princeton
  university Press, New Jersey, 2005.

\bibitem{Fok1997a}
A.~S. Fokas, \emph{A unified transform method for solving linear and certain
  nonlinear {PDE}s}, Proc. R. Soc. Lond. Ser. A Math. Phys. Eng. Sci.
  \textbf{453} (1997), 1411--1443.

\bibitem{FP2001a}
A.~S. Fokas and B.~Pelloni, \emph{Two-point boundary value problems for linear
  evolution equations}, Math. Proc. Cambridge Philos. Soc. \textbf{131} (2001),
  521--543.

\bibitem{FP2005b}
\bysame, \emph{Boundary value problems for {B}oussinesq type systems}, Math.
  Phys. Anal. Geom. \textbf{8} (2005), no.~1, 59--96.

\bibitem{FS2013a}
A.~S. Fokas and D.~A. Smith, \emph{Evolution {P}{D}{E}s and augmented
  eigenfunctions. {F}inite interval}, Adv. Diff. Eq. (2016), (to appear)
  arXiv:1303.2205 [math.SP].

\bibitem{FS2012a}
A.~S. Fokas and E.~A. Spence, \emph{Synthesis, as opposed to separation, of
  variables}, SIAM Rev. \textbf{54} (2012), no.~2, 291--324.

\bibitem{FS1999a}
A.~S. Fokas and L.~Y. Sung, \emph{Initial-boundary value problems for linear
  dispersive evolution equations on the half-line}, (unpublished), 1999.

\bibitem{GS1967a}
I.~M. Gel'fand and G.~E. Shilov, \emph{Generalized functions volume 3: theory
  of differential equations}, Academic Press, 1967, Trans. M. E. Mayer from
  Russian (1958).

\bibitem{GV1964a}
I.~M. Gel'fand and N.~Ya. Vilenkin, \emph{Generalized functions volume 4:
  applications of harmonic analysis}, Academic Press, 1964, Trans. A. Feinstein
  from Russian (1961).

\bibitem{Pel2004a}
B.~Pelloni, \emph{Well-posed boundary value problems for linear evolution
  equations on a finite interval}, Math. Proc. Cambridge Philos. Soc.
  \textbf{136} (2004), 361--382.

\bibitem{Pel2005a}
\bysame, \emph{The spectral representation of two-point boundary-value problems
  for third-order linear evolution partial differential equations}, Proc. R.
  Soc. Lond. Ser. A Math. Phys. Eng. Sci. \textbf{461} (2005), 2965--2984.

\bibitem{Smi2012a}
D.~A. Smith, \emph{Well-posed two-point initial-boundary value problems with
  arbitrary boundary conditions}, Math. Proc. Cambridge Philos. Soc.
  \textbf{152} (2012), 473--496.

\bibitem{Smi2014a}
\bysame, \emph{The unified transform method for linear initial-boundary value
  problems: a spectral interpretation}, Unified transform method for boundary
  value problems: applications and advances (A.~S. Fokas and B.~Pelloni, eds.),
  SIAM, Philadelphia, PA, 2015, pp.~34--48.

\end{thebibliography}

\end{document}